\documentclass[a4paper,12pt]{amsart}
\usepackage{amsmath}
\usepackage{amsfonts}
\usepackage{amssymb} 
\usepackage{latexsym}
\usepackage{color}

\usepackage[backref=true]{hyperref}

\definecolor{webred}{cmyk}{0,1,1,.3}
\definecolor{webgreen}{cmyk}{1,0,1,.5}
\definecolor{webblue}{cmyk}{1,1,0,.5}

\def\E{\mathcal{E}}

\def\Sym{{\sf Sym}\,}
\def\soc{{\rm soc}\,}
\def\al{\alpha}

\renewcommand{\wr}{\,{\sf wr}\,}

\newcommand{\dih}[1]{{\sf D}_{#1}}
\newcommand{\alt}[1]{{\sf A}_{#1}}
\newcommand{\mat}[1]{{\sf M}_{#1}}
\newcommand{\pomegap}[2]{{\sf P}\Omega^+(#1,#2)}
\newcommand{\psl}[2]{\mbox{\sf PSL}({#1},#2)}
\newcommand{\pgl}[2]{\mbox{\sf PGL}({#1},#2)}
\newcommand{\psu}[2]{\mbox{\sf PSU}({#1},#2)}

\newcommand{\comp}[2]{#1^{(#2)}}

\renewcommand{\sp}[2]{{\sf Sp}(#1,#2)}

\newcommand{\ver}[1]{{\sf V}(#1)}
\newcommand{\ed}[1]{{\sf E}(#1)}
\newcommand{\aut}[1]{{\sf Aut}\,#1}
\newcommand{\kr}[2]{#1^{[1]}_{#2}}

\newcommand{\X}{\mathcal X}

\renewcommand{\S}{\mathcal S}

\newcommand{\gl}[2]{\mbox{\sf GL}(#1,#2)}
\newcommand{\SL}[2]{\mbox{\sf SL}(#1,#2)}
\newcommand{\pgammal}[2]{\mbox{\sf P$\Gamma$L}(#1,#2)}

\newcommand{\cd}[1]{{\sf CD}_{#1}}
\newcommand{\Z}{\mathbb Z}

\newcommand{\cent}[2]{{\sf C}_{#1}(#2)}
\newcommand{\norm}[2]{{\sf N}_{#1}(#2)}

\newcommand{\sym}[1]{{\sf Sym}\,#1}
\newcommand{\sy}[1]{{\sf S}_{#1}}

\newtheorem{theorem}{Theorem}[section]
\newtheorem{lemma}[theorem]{Lemma}
\newtheorem{corollary}[theorem]{Corollary}
\newtheorem{proposition}[theorem]{Proposition}

\newtheorem{claim}{Claim}

\theoremstyle{definition}

\newtheorem{hypothesis}[theorem]{Hypothesis}

\renewcommand{\leq}{\leqslant}
\renewcommand{\geq}{\geqslant}

\begin{document}
 
\title[Wreath products in product action and 2-arc-transitive graphs]
{Inclusions of innately transitive groups into 
wreath products in product action with applications to 
$2$-arc-transitive graphs}

\author{Cai-Heng Li}
\address[C. H. Li]{School of Mathematics and Statistics\\
The University of Western Australia\\
35 Stirling Highway 6009 Crawley\\
Western Australia\\
 cai.heng.li@uwa.edu.au \newline
www.maths.uwa.edu.au/$\sim$li}

\author{Cheryl E. Praeger}
\address[C. E. Praeger]{School of Mathematics and Statistics\\
The University of Western Australia\\
35 Stirling Highway 6009 Crawley\\
Western Australia\\
 cheryl.praeger@uwa.edu.au\newline
www.maths.uwa.edu.au/$\sim$praeger}

\author{Csaba Schneider}
\address[C. Schneider]{Departamento de Matem\'atica\\
Instituto de Ci\^encias Exatas\\
Universidade Federal de Minas Gerais\\
Av.\ Ant\^onio Carlos 6627\\
Belo Horizonte, MG, Brazil\\
csaba@mat.ufmg.br\\
 www.mat.ufmg.br/$\sim$csaba}

\begin{abstract}
We study $(G,2)$-arc-transitive graphs for innately transitive permutation 
groups $G$ such that $G$ can be embedded into a wreath product 
$\sym\Gamma\wr\sy\ell$ acting in product action on $\Gamma^\ell$. We find two
such connected
graphs: the first is Sylvester's double six graph with 36 vertices, while
the second is a graph with $120^2$ vertices whose automorphism
group is $\aut\sp 44$. We prove that under certain conditions no more
such graphs exist.
\end{abstract}

\maketitle
 
\section{Introduction}
Graphs and groups in the paper are finite; further our graphs are 
undirected, simple, and connected.
Graphs that satisfy certain symmetry type properties have been extensively 
studied, especially the class of 
$s$-arc-transitive graphs, for $s\geq 2$.
An $s$-arc in a graph $\Gamma$ is a sequence of vertices
$(v_0,\ldots,v_{s})$ such that $v_i$ and $v_{i+1}$ are neighbors
for all $0\leq i\leq s-1$ and $v_i\neq v_{i+2}$ for $0\leq i\leq s-2$. 
If $G$ is a subgroup of the automorphism group of $\Gamma$ that acts transitively on 
the set of $s$-arcs of $\Gamma$, then we say that $\Gamma$ is 
{\em $(G,s)$-arc-transitive}.
The graph $\Gamma$ is said to be {\em $s$-arc-transitive}, if it is 
$(\aut\Gamma,s)$-arc-transitive.
The values of $s$ for which this property can hold are very restricted:
by the celebrated results of Weiss and Tutte, if $\Gamma$ is an 
$s$-arc-transitive graph with valency at least 3, then $s\leq 7$~\cite{weiss}, 
while if $\Gamma$ is a cubic
graph then $s\leq 5$~\cite{tutte}.
The classification problem
of $s$-arc-transitive graphs has a long history; see for 
instance~\cite{seress} for an excellent survey.

In the theory of 2-arc-transitive graphs, normal quotients play an important
role. If $\Gamma$ is a $(G,2)$-arc-transitive graph, and $N$ is a
normal subgroup of $G$ that is intransitive on the vertex set of 
$\Gamma$, then we construct a quotient graph 
$\Gamma_N$ whose vertices are the $N$-orbits. If the original graph $\Gamma$ 
is not bipartite, then  
$\Gamma_N$ is
$(G/N,2)$-arc-transitive. Further, in this case, $\Gamma$ is a normal
cover of $\Gamma_N$; see~\cite[Theorem~4.1]{prae:quasi} or
\cite[Lemma~3.1]{seress}.
Hence it is important to study $(G,2)$-arc-transitive 
graphs that do not admit
 proper normal quotients.
In this case,
all non-trivial normal subgroups of $G$ 
are transitive and $G$ is said to be {\em quasiprimitive}; see~\cite{prae:quasi}.
Thus the $(G,2)$-arc-transitive graphs  for quasiprimitive
permutation groups $G$ can be considered as building blocks for 
finite, non-bipartite, 2-arc-transitive graphs, and they require
special attention. However, our understanding
 of the class of $(G,2)$-arc-transitive graphs
is still far from complete for a general quasiprimitive groups $G$. 

Quasiprimitive groups were classified by the second author~\cite{prae:quasi} 
via a type-analysis similar to the
O'Nan--Scott classification of primitive groups. 
She showed that if $\Gamma$ is a $(G,2)$-arc-transitive graph 
with a quasiprimitive group $G$, then the O'Nan--Scott type of $G$ 
is one of four types, denoted {\sc As}, HA, {\sc Pa} 
and {\sc Tw} (see~\cite{bad:quasi} for
the definitions of these O'Nan--Scott types).
All examples for type HA were determined in~\cite{iva-prae} 
and those of type TW were studied in~\cite{baddeleytw1}. 
In addition many examples of type {\sc As} are known; 
see, for example, \cite{ree,suzuki2at,psl2at}. On 
the other hand, it was not known for more than 20 years after the 
publication of~\cite{prae:quasi} whether there 
were examples of type {\sc Pa}. The first {\sc Pa} examples were constructed by 
Seress and the first author in~\cite{MR2258005}, 
and links between the {\sc Pa} and {\sc As} types were explored in a
recent paper; see~\cite[Theorem 1.3]{liseresssong}.

If a group $G$ of type PA is primitive, then the underlying 
set can be identified with a cartesian product 
$\Gamma^\ell$, for $\ell\geq 2$, in such a way that $G$ 
is a subgroup of the wreath product $W=\Sym 
\Gamma \wr \sy\ell$ acting in product action on $\Gamma^\ell$; 
moreover $G$ has socle $T^\ell$, for 
a non-abelian simple group $T$, and the component induced by $G$ on 
$\Gamma$ is primitive of 
AS type with socle $T$. 
In our terminology introduced in Section~\ref{sect:normalinc},
the embedding $G\leq W$ 
is a normal inclusion. It is 
possible, but not always the case, that an imprimitive quasiprimitive group 
$G$ has a normal inclusion 
into a wreath product in product action. However, 
for none of the five infinite families of 
examples of $(G,2)$-arc-transitive graphs in~\cite{MR2258005}, 
with $G$ quasiprimitive of type {\sc Pa}, does $G$ have such a normal inclusion. 

In this paper we seek, and give, answers to the following questions:

\begin{enumerate}
\item Are there any $(G,2)$-arc transitive graphs, 
with $G$ quasiprimitive of type {\sc Pa}, such that $G$
has a {\bf normal inclusion} in a wreath product in product action 
on the vertex set?
\item More generally, do there exist any $(G,2)$-arc transitive graphs 
with $G$ quasiprimitive such 
that $G$ can be embedded in a wreath product in product action on the 
vertex set?  In particular, is this true for any of the examples 
in~\cite{MR2258005}?
\end{enumerate}

Our results answer both questions, yet raise an interesting open problem. 
First we give a negative 
answer to Question 1 (Corollary~\ref{nonnormcor}), and 
in Section~\ref{sect:ls} we analyze the examples in~\cite{MR2258005} and prove 
(Corollary~\ref{cor:ls}) that for none of them does the group $G$ 
embed in a wreath product in product action. 
The latter conclusion relies on our general analysis of the situation for 
Question~2. Our main result 
Theorem~\ref{main} proves that there do indeed exist at least 
two examples, with each of these examples 
having type {\sc As}, that is the group $G$ is almost simple.

Using the terminology of~\cite{prae:incl,bad:quasi}, 
we say that the containment $G\leq W$ is an \emph{inclusion 
of $G$ into $W$}, and if $W=\Sym \Delta\wr S_\ell$ 
in product action on $\Delta^\ell$, then $W$ is the 
full stabilizer of the natural cartesian decomposition of 
$\Delta^\ell$ (see Section~\ref{sect:cs}), and hence a 
subgroup $G$ of $W$ also stabilizes this decomposition. 
Our analysis for Question 2 builds on the 
theory developed by the second and third authors in a
 series of papers~\cite{cs,MR2186989,MR2309892,MR2342458,MR2346470,embedding} 
with Baddeley 
to describe the cartesian decompositions stabilized by a 
given \emph{innately transitive} group, namely 
a permutation group with a transitive minimal normal subgroup $N$.  
Such a subgroup $N$ is called a 
\emph{plinth}.  Since each quasiprimitive group is 
innately transitive, we make our analysis in the 
broader context of innately transitive groups.  
Moreover, since the case of abelian plinths is completely 
settled in Proposition~\ref{lem6}
using the results in~\cite{iva-prae}, we usually
assume that the plinth is non-abelian. Further, 
since for quasiprimitive type {\sc Pa} the plinth is 
not regular, we assume this also.  The type $\cd{2\not\sim}$ 
for an inclusion $G\leq W$ is defined in Section~\ref{sect:transcd}.

\begin{theorem}\label{main}
Suppose that $\Delta$ is a finite set with $|\Delta|\geq 2$ 
and let $\ell\geq 2$. 
Consider $W=\sym\Delta\wr\sy\ell$ as a permutation group acting 
in product action on $\Omega=\Delta^\ell$. Let $\Gamma$ be a connected 
$(G,2)$-arc-transitive 
graph with vertex set $\Omega$ for 
some innately transitive group $G\leq W$ with non-regular plinth $M$.
Then one of the following is valid.
\begin{enumerate}
\item $G$ is quasiprimitive of type {\sc As} and either
\begin{enumerate}
\item $M\cong\alt 6$, $|\aut \alt 6:G|\in\{1,2\}$,  $G\neq\pgl 29$,
$|\Omega|=6^2$, and $\Gamma$ has valency $5$; or
\item $M\cong\sp 44$, $G=\aut\sp 44$, $|\Omega|=120^2$
and $\Gamma$ is a graph with valency $17$.
\end{enumerate}
\item $M$ non-abelian and  non-simple, $G$ projects onto 
a transitive subgroup of $\sy\ell$ and 
the inclusion $G\leq W$ is of type $\cd {2\not\sim}$.
\end{enumerate}
\end{theorem}

There is a unique example for each part of Theorem~\ref{main}(1).
These  two graphs are, to our knowledge,
the first known connected $(G,2)$-arc transitive graphs
with $G$ quasiprimitive and
preserving a cartesian decomposition of the vertex set.
The graph in part~(a) is 
known as Sylvester's Double Six Graph 
(see~\cite[13.1.2~Theorem]{distancebook}). Both  graphs in 
part~(1) can be 
constructed using the generalized quadrangle associated with 
the non-degenerate alternating bilinear form preserved by $\sp 4q$
with $q=2,\ 4$.  
On the other hand, the existence of examples in part (2) is unresolved.

\medskip
\noindent{\bf Problem.}  Does there exist a $(G,2)$-arc transitive 
graph $\Gamma$ for which Theorem~\ref{main}(2) holds?

\medskip

The class of inclusions of type $\cd{2\not\sim}$ is discussed in 
Section~\ref{sect:transcd}. Of the six types of inclusions 
characterized there, it is the one for which least information is available. 
So it is not surprising that for 
these inclusions our results are, unfortunately, inconclusive.  
The theory of cartesian decompositions invariant under the action of 
an innately transitive group is 
spread over several papers, and is presented in terms of 
cartesian systems of subgroups of the plinth. 
We found that the language of inclusions into wreath products 
was better suited for our study of 2-arc-transitive graphs, 
and we believe that this is probably true also for 
other applications in algebraic 
combinatorics and graph theory.  For this reason we summarize the most 
important aspects of the 
theory in Section~\ref{sect:cs} using the language of inclusions.  
In particular, if $G\leq W = \Sym \Delta \wr \sy\ell$, 
acting in product action on $\Delta^\ell$, with $G$ innately transitive, 
then the image $G\pi$ under the natural 
projection map $\pi:W\rightarrow \sy\ell$ has at most two orbits 
on $\underline \ell=\{1,\ldots,\ell\}$. 
The groups
that have two orbits can be characterized, and the groups that 
are transitive on $\underline\ell$ are classified into six~classes. 
 For all classes except for $\cd{2\not\sim}$, the class that appears in 
Theorem~\ref{main}(2), 
the information that we have about $G$ is so strong that
its permutational isomorphism type can be described in quite precise details
 (see Theorem~\ref{c2}).
The results from Section~\ref{sect:cs} are 
applied in Sections~\ref{cdgraphs}--\ref{sect:ls} 
to prove Theorem~\ref{main} and to analyze the graphs in~\cite{MR2258005}. 
We believe that this 
summary will make the results more accessible for other researchers. 
For example, the study by Morris 
and Spiga in~\cite{joyspiga1} 
significantly generalized results about distance-transitive graphs 
to $(G,\Lambda)$-transitive digraphs  for which the group $G$ 
is innately transitive and has a normal inclusion into a wreath 
product in product action. Perhaps using the framework we
present in Section~\ref{sect:cs}, their results could be further
extended to describe not necessarily normal inclusions for such graphs.

\section{Inclusions of innately transitive groups into wreath products}
\label{sect:cs}

We introduce notation that we will 
keep throughout the section. For a natural number $n$, the symbol
$\underline n$ denotes the set $\{1,\ldots,n\}$. We will work under the following 
hypothesis.
\begin{hypothesis}\label{hyp1}
Suppose that $\Delta$ is a finite set with 
$|\Delta|\geq 2$, $\ell\geq 2$, 
 and let $\Omega=\Delta^\ell$. 
Consider $W=\sym\Delta\wr\sy\ell=(\sym\Delta)^\ell\rtimes\sy\ell$ 
as a permutation group acting in product
action on $\Omega$. Suppose that $G$ is subgroup of $W$ and $M$ is 
a minimal normal subgroup of $G$ that is transitive on $\Omega$. Let 
$M=T_1\times\cdots\times T_k\cong T^k$ where the $T_i$ and $T$ are
finite simple groups.
\end{hypothesis}

By the definition introduced in the introduction,
$G$ is an innately transitive group on $\Omega$ with plinth $M$.
For $(\delta_1,\ldots,\delta_\ell)\in\Omega$, $g_1,\ldots,g_\ell\in
\sym\Delta$ and $h\in\sy\ell$ the product action of $(g_1,\ldots,g_\ell)h\in W$
on $(\delta_1,\ldots,\delta_\ell)$ can be written as
$$
(\delta_1,\ldots,\delta_\ell)(g_1,\ldots,g_\ell)h=(\delta_{1h^{-1}}g_{1h^{-1}},
,\ldots,\delta_{\ell h^{-1}}g_{\ell h^{-1}}).
$$
Inclusions of primitive and quasiprimitive groups into wreath 
products in product action were described in~\cite{prae:incl,bad:quasi}.
Innately transitive subgroups
of wreath products in product action were studied 
in~\cite{cs,MR2186989,MR2309892,MR2342458,MR2346470}, 
by analyzing cartesian decompositions. A cartesian decomposition of $\Omega$
is a set $\E=\{\Delta_1,\ldots,\Delta_\ell\}$ of partitions of $\Omega$ such that
$$
|\delta_1\cap\cdots\cap\delta_\ell|=1\quad\mbox{for all}
\quad\delta_1\in\Delta_1,\ldots,\delta_\ell\in\Delta_\ell.
$$
In our case $W$ is the full stabilizer in $\sym\Omega$ of the
{\em natural
cartesian decomposition} $\E$ of $\Omega$ where $\E$ is defined by
$$
\E=\{\Delta_i\mid i\in\underline\ell\}
\quad\mbox{where}\quad \Delta_i=\{\{(\delta_1,\ldots,\delta_\ell)\mid
\delta_i=\delta\}\mid\delta\in\Delta\}
$$
(see~\cite[Section~1.1]{embedding}).
The cartesian decomposition $\E$ contains $\ell$ partitions
of $\Omega$, and 
$(\delta_1,\ldots,\delta_\ell)$ and 
$(\gamma_1,\ldots,\gamma_\ell)$ belong to the same block of $\Delta_i$ if and only 
if $\delta_i=\gamma_i$. The group $W$ preserves $\E$ in the sense that
it permutes the partitions in $\E$. 

The aim of this section is to combine the theory of cartesian decompositions
with the Embedding Theorem in~\cite{embedding} to 
derive useful facts on the inclusion of $G$ into $W$ that can be used
in the characterization of $(G,2)$-arc-transitive graphs in 
Sections~\ref{cdgraphs}--\ref{sec:nonsimple}.
The main results do not refer to $\E$ directly,
but the fact that $W$ and $G$ preserve $\E$
plays a central role in their proofs. For the key assertions we
give references to the appropriate results in the papers cited above.

Suppose that $\pi:W\rightarrow\sy\ell$ 
is the natural projection and consider $\pi$ as a permutation 
representation of $W$ on $\underline\ell$. The kernel of $\pi$ is the
base group $(\sym\Delta)^\ell$ of the wreath product $W$.
If $i,\ j\in\underline\ell$, then
we have, for all $w\in W$, that $i(w\pi)=j$ if and only if 
$\Delta_iw=\Delta_j$. Thus the $W$-actions on $\underline\ell$ and 
on $\E$ are permutationally equivalent. This permutational equivalence
implies that the
pointwise stabilizer in $W$ of $\E$ is the base group $(\sym\Delta)^\ell$ of 
the wreath product. Combining this with~\cite[Proposition~2.1]{cs} 
gives the following.

\begin{lemma}\label{minbase}
Assuming Hypothesis~\ref{hyp1}, $M\leq(\sym\Delta)^\ell$. 
\end{lemma}

In~\cite{embedding}, 
we defined, for $j\in\underline\ell$, the $j$-th component of 
the group $G$. 
Let us express this component in terms of the inclusion
$G\leq W$.
Let $W_j$ denote the stabilizer of $j$ under $\pi$.
The group $W_j$ can be written as a
direct product 
\begin{equation}\label{Wjdecomp}
W_j=\sym\Delta\times(\sym\Delta\wr\sy{\ell-1})
\end{equation} 
where
the first factor acts on the $j$-th coordinate of $\Omega=\Delta^\ell$ and
the second factor acts on the other $\ell-1$ coordinates. In particular
`$\sy{\ell-1}$' is meant to be the stabilizer of $j$ in $\sy\ell$. We define
the $j$-th component $\comp Gj$ of $G$ as the projection of $G_j=G\cap W_j$ 
into the first factor of the direct product decomposition of $W_j$
given in~\eqref{Wjdecomp}.
It is easy to see that this definition of the component is equivalent 
to the one in~\cite[Section~1.2]{embedding}.

As both $G$ and $M$ are transitive on $\Omega$, 
the following is a direct consequence of Theorems~1.1--1.2 of~\cite{embedding}.

\begin{theorem}\label{comps}
Suppose that Hypothesis~\ref{hyp1} is valid and let $\omega\in\Omega$.
\begin{enumerate}
\item The components $\comp Gj$  and $\comp Mj$  are transitive on $\Delta$
for all $j\in\underline\ell$.
\item There is an element $x\in(\sym\Delta)^\ell$ stabilizing $\omega$ such that 
the components of $G^x$ are constant on each $G\pi$-orbit in $\underline\ell$.
\item If $G\pi$ is transitive on $\underline\ell$ then the element
$x$  in part~(2) can be chosen so that $G^x\leq \comp G1\wr(G\pi)$. 
\end{enumerate}
\end{theorem}

The inclusion of $G$ into $W$ can be further studied by understanding the
components of $M$. When $M$ is abelian, then $G$ is primitive of HA type.
The fact that this case does not occur for a $(G,2)$-arc-transitive graph
is proved in Proposition~\ref{lem6} using the tools developed for primitive
inclusions in~\cite{prae:incl}. Thus from now on we assume that 
$M$ is non-abelian.

\subsection{Non-abelian $\boldsymbol M$}
Suppose that $M$ is non-abelian. In this case, $M=T^k=T_1\times
\cdots\times T_k$ where the $T_i$ are the minimal normal subgroups of $M$. 
The $j$-th component $\comp Mj$ is, by definition, a quotient of $M$. 
Suppose that $\overline{\comp Mj}$ is the kernel of the natural 
projection
$M\rightarrow \comp Mj$. Then $\overline{\comp Mj}$ is the direct product of 
some of the $T_i$. We will identify $\comp Mj$ with the product of the
$T_i$ such that $T_i\not\leq \overline{\comp Mj}$. Hence, 
$M=\overline{\comp Mj}\times \comp Mj$ and $\comp Mj$ acts faithfully
on $\Delta$ viewed as the $j$-th coordinate of $\Delta^\ell$. 
Let $T$ be a simple factor of $M$ 
and set $\overline T=\cent MT$, so that $M=T\times\overline T$.
If $H$ is a subgroup of $M$, then 
$H\overline T/\overline T$ is the projection of $H$ into $T$, and we identify
this quotient with a subgroup of $T$.
The following lemma states important factorization properties of
the projections $(\comp Mj)_\delta\overline T/\overline T$,
where $\delta\in\Delta$,
for those components  $(\comp Mj)$
that contain $T$.


\begin{lemma}\label{notmax}
Suppose that Hypothesis~\ref{hyp1} holds, that $M$ is non-abelian, and that $T$ is
a simple factor of $M$. Suppose that 
$T$ is contained in the components $\comp M{j_1},\ldots, \comp M{j_s}$. Let $\delta\in\Delta$, set $\omega=(\delta,\ldots,\delta)$, and 
for $r\in\underline s$,
let $A_{r}=(\comp M{j_r})_\delta\overline T/\overline T$.  
Then 
$$
A_r\left(\bigcap_{i\neq r}A_i\right)=T\mbox{ for all $r$ and }
M_\omega\overline T/\overline T\leq \bigcap_{r\leq s}A_r.
$$
\end{lemma}
\begin{proof}
We introduced, in~\cite{cs}, 
a cartesian system of subgroups  associated with a cartesian 
decomposition. The cartesian system of subgroups with 
respect to $\omega=(\delta,\ldots,\delta)\in\Omega$ for the natural
cartesian decomposition $\E$ of $\Omega=\Delta^\ell$ is defined as 
$\{K_1,\ldots,K_\ell\}$ where $K_j$ is the stabilizer in $M$ of the set
$$
\{(\delta_1,\ldots,\delta_\ell)\mid\delta_j=\delta\}.
$$
Simple consideration shows that 
$K_j=(\comp Mj)_\delta\times\overline{\comp Mj}$.
Now the result follows by
applying the projection $M\rightarrow T$ to the equations in~\cite[Definition~1.3]{cs}.
\end{proof}
The information given by the previous lemma is very strong. 
Using the terminology of~\cite{bad:fact}, 
the set of proper subgroups among $A_1,\ldots,A_s$ forms a strong multiple 
factorization for the non-abelian finite simple group $T$.
The fact that we understand such factorizations in sufficient details leads
to the structure theorem in Section~\ref{sect:strth}.

\subsection{Normal inclusions}\label{sect:normalinc}
Assume that $M$ is non-abelian and identify 
each projection $\comp Mj$ with a normal 
subgroup of $M$ as above.
The inclusion $G\leq W$ is said to be {\em normal} if
$M=\prod_j \comp Mj$. 
Since our definition of the component $\comp Mj$ is equivalent to the
one given in~\cite{MR2309892},  an inclusion $G\leq W$ is normal  if and only if
the natural cartesian decomposition $\E$ of $\Delta^\ell$ is $M$-normal, as defined
in~\cite{MR2309892}. The inclusion $G\leq W$ is normal if and only if 
for all $T_i$ there is a unique $j$ such that $T_i\leq
\comp Mj$; that is, each $T_i$ acts trivially on all but one
coordinate of 
$\Delta^\ell$. The following proposition is a direct consequence 
of~\cite[Lemma~3.1]{MR2309892}.

\begin{proposition}\label{normalstab}
Suppose Hypothesis~\ref{hyp1} and let $\omega=(\delta,\ldots,\delta)$ with 
some $\delta\in\Delta$.
If $G\leq W$ is a normal inclusion, then $G\pi$ is transitive on 
$\underline\ell$ and $M_\omega= \prod_j (\comp Mj)_\delta$. 
\end{proposition}

The following result gives a criterion to recognize
possible normal inclusions $G\leq W$.
 For a group 
$X$, a set $\X=\{X_1,\ldots,X_r\}$ of subgroups of $X$ is said to be 
a {\em direct decomposition} of $X$ if $X=X_1\times\cdots\times X_r$. A subgroup
$Y\leq X$ is said to be an {\em $\X$-subgroup}, if $\{Y\cap X_i\mid i\in\underline r\}$
is a direct decomposition of $Y$. When $G\leq W$ is a normal inclusion,
then $M_\omega\cap \comp Mj=(\comp Mj)_\delta$, and so 
Proposition~\ref{normalstab} states
that
$\X=\{\comp Mj\mid j\in\underline\ell\}$ is a direct decomposition of $M$ such
that $M_\omega$ is an $\X$-subgroup. We show a converse to this 
statement in  the next theorem.

\begin{theorem}\label{constrnormal}
Suppose that $G$ is a finite innately transitive permutation group 
acting on $\Omega$ with a non-abelian plinth $M$ and let $\omega\in\Omega$. 
Suppose that there exists a $G$-invariant 
direct decomposition $\X=
\{M_1,\ldots,M_r\}$ of $M$ such that $M_\omega$ is an $\X$-subgroup.
Let $\Xi$ be the right coset space $[M_1:M_1\cap M_\omega]$. Then 
there exists a monomorphism $\alpha:G\rightarrow \Sym\Xi\wr\sy r$,
where the wreath product is taken as a permutation group acting in product
action on $\Xi^r$,
in such a way that the inclusion $G\alpha\leq \Sym\Xi\wr\sy r$ is a normal inclusion.
\end{theorem}
\begin{proof}
Given a direct product decomposition of $M$ as above, one can construct 
a normal cartesian decomposition $\E'$ preserved by $G$ as 
in~\cite[Example~7.1]{MR2186989}. The full stabilizer of $\E'$ is 
isomorphic to $\sym\Xi\wr\sy r$ in product action 
(see~\cite[Section~1.2]{embedding}) and $G$ can be
embedded into  this wreath product using a monomorphism $\alpha$. 
Since $\E'$ is $M$-normal, the inclusion $G\alpha\leq \sym\Xi\wr\sy r$
is normal.
\end{proof}

In the case when $G$ is primitive, 
the direct product decomposition $\{M_1,\ldots,M_r\}$ in the theorem
is referred to as a blow-up decomposition by~\cite{kov:blowup}.

\begin{lemma}
Assuming~Hypothesis~\ref{hyp1}, if $G\leq W$ is a normal inclusion, then $\comp Mj$ is a transitive 
minimal normal subgroup of $\comp Gj$. Further, if $M$
is non-regular, $\comp Mj$ is the unique such transitive minimal normal 
subgroup of $\comp Gj$. 
\end{lemma}
\begin{proof}
Since $M\unlhd G$, $\comp Mj\unlhd\comp Gj$. 
Theorem~\ref{comps} implies that $\comp Mj$ is transitive. 
Since $M$ is non-abelian, to prove that $\comp Mj$ is a 
minimal normal subgroup of $\comp Gj$, it suffices to show that
$\comp Gj$ permutes transitively the simple components of $\comp Mj$. 
Suppose, without loss of generality, that $T_1$ and $T_2$ are 
simple components of $\comp Mj$. Then there is some $g\in G$ 
such that $(T_1)^g=T_2$. Since $G\leq W$ is a normal inclusion, $\comp Mj$ is the
unique component that contains $T_1$ and $T_2$. Hence, 
$g\pi\in W_j$. 
Thus, using the decomposition in~\eqref{Wjdecomp},
 $g=g_1g_2$ where $g_1\in\comp Gj$ and 
$g_2\in(\sym\Delta)\wr\sy{\ell-1}$. The component $g_2$ centralizes 
both $T_1$ and $T_2$, and hence we must have that $(T_1)^{g_1}=T_2$. Thus
$\comp Gj$ is transitive on the simple components of $\comp Mj$, and 
hence $\comp Mj$ is a minimal normal subgroup of $\comp Gj$. 

By Proposition~\ref{normalstab}, 
$$
|\Omega|=|M:M_\omega|=\prod_{j=1}^\ell|\comp Mj:(\comp Mj)_\delta|=\frac{M}{\prod_j|(\comp Mj)_\delta|}.
$$
Hence $M$ is regular if and only if each component $\comp Mj$ is regular.
In particular, if $M$ is not regular, then neither is $\comp Mj$ and 
then it must be the unique minimal normal subgroup of $\comp Gj$.
\end{proof}

\subsection{Transitive $\boldsymbol {G\pi}$}\label{sect:transcd}
We now turn our attention to non-normal 
inclusions. Assume that $G\pi$ is transitive on $\underline\ell$.
This implies that the conjugation action of $G$ is transitive on 
the set $\{\comp Mj\}$.
Set $T=T_1$. Since $G$ acts transitively on
the $T_i$ by conjugation, the treatment that follows does not depend 
on this choice of  $T$.
The group $T$ can be
a subgroup in several of the components $\comp Mj$. Since 
$G$ acts transitively on both sets $\{T_i\}$ and $\{\comp Mj\}$, the
number of components $\comp Mj$ such that $T\leq\comp Mj$ 
is independent of the choice of  $T$.
The inclusion $G\leq W$ is normal if and only if this number is one. 
We will investigate here the situation when this is not
the case; that is, $T$ is contained in at least two $\comp Mj$. 
As above, we set $\overline T=\cent MT$.




\begin{theorem}\label{types}
Suppose that Hypothesis~\ref{hyp1} holds, that $M$ is non-abelian, and that   
$G\pi$ is transitive. Suppose that $T$ is a simple factor of $M$. 
Then 
the following are valid.
\begin{enumerate}
\item The number $s$ of  components $\comp Mj$ such that 
$T\leq \comp Mj$ is independent of the choice of $T$ and
$s\leq 3$.
\item The inclusion $G\leq W$ is normal if and only if $s=1$.
\item If  $(\comp Mj)_\delta\overline T/\overline T=T$ for some $j$ 
and $\delta\in\Delta$ then
$s\leq 2$.
\end{enumerate}
\end{theorem}
\begin{proof}
As explained in the proof of Lemma~\ref{notmax}, 
the cartesian system of subgroups with respect to 
$\omega=(\delta,\ldots,\delta)$ that corresponds
to the natural cartesian decomposition of $\Omega$ is 
$\{K_1,\ldots,K_\ell\}$ where $K_j=(\comp Mj)_\delta\times \overline{\comp Mj}$. 
Now parts~(1) and~(3) follow from~\cite[Theorem~6.1]{MR2186989}. Part~(2)
was justified before the theorem.
\end{proof}

 We use Theorem~\ref{types} 
to define the inclusion types $\cd {1S}$, $\cd {2\sim}$, $\cd{2\not\sim}$,  
and $\cd 3$ for  inclusions $G\leq W$ in~Hypothesis~\ref{hyp1}
such that $T$ is non-abelian and $G\pi$ is transitive on $\underline\ell$.
If, in Theorem~\ref{types}, $s=1$, 
the inclusion $G\leq W$ is normal, and so we assume that this is not 
the case. Set $\delta\in\Delta$.
\begin{description}
\item[$\boldsymbol{\cd{1S}}$]
  We say that the inclusion  $G\leq W$ is of type $\cd{1S}$ if $s=2$ and 
$(\comp Mj)_\delta\overline T/\overline T=T$ for some $j$. 
\item[$\boldsymbol{\cd 2}$] Suppose that $s=2$ and  
$(\comp Mj)_\delta\overline T/\overline T\neq T$ for all $j$. 
Let $j_1$ 
and $j_2$ be distinct indices such that $T\leq \comp M{j_1}\cap \comp M{j_2}$.
Then the coordinate projections $A=(\comp M{j_1})_\delta\overline T/\overline T$ 
and $B=(\comp M{j_2})_\delta\overline T/\overline T$ are proper 
subgroups of $T$.
We say that the inclusion $G\leq W$ is 
of type $\cd{2\sim}$ if $A\cong B$. Otherwise the inclusion 
is of type $\cd{2\not\sim}$.
\item[$\boldsymbol{\cd 3}$] We say that the inclusion $G\leq W$ is of type $\cd 3$ if $s=3$.
\end{description}

The fact that these types are well-defined was proved
in~\cite[Theorem~6.2]{MR2186989}.
The definition of the inclusion type $\cd{2\sim}$ is slightly different 
in~\cite{MR2186989} where we required that the two proper projections
should actually be $G_\omega$-conjugates. The definition given here leads
to a stronger theorem.

Suppose that $G\leq W$ is an inclusion of 
type $\cd{1S}$ and fix $\delta\in\Delta$. 
Suppose that $T$ is a simple factor of $M$ such that
$T\leq\comp Mj$ and $(\comp Mj)_\delta\overline T/\overline T=T$. 
By a version of Scott's Lemma~\cite[Lemma~4.2]{MR2186989} 
on subdirect subgroups in 
direct products of non-abelian simple groups,  
$(\comp Mj)_\delta$ contains a full diagonal subgroup isomorphic to $T$.
Then $(\comp Mj)_\delta=D\times\cent{(\comp Mj)_\delta}D$ where
$D$ is a full diagonal subgroup in $T\times T_i$ for some $i$; that is,
$D=\{(t,t\alpha)\mid t\in T\}$ with some isomorphism $\alpha:T\rightarrow 
T_i$. We say that the full diagonal subgroup $D$ is 
{\em involved in $(\comp Mj)_\delta$} and that $\{T,T_i\}$ is the 
{\em support} of 
$D$. In this case, $k$ is even and $M$ admits a direct decomposition 
\begin{equation}\label{dd1s}
\mathcal S=\{M_1,\ldots,M_{k/2}\}
\end{equation}
such that $M_i=T_{i_1}\times T_{i_2}$ where $\{T_{i_1},T_{i_2}\}$ is the
support of some full diagonal subgroup $D$ involved in $(\comp Mj)_\delta$
for some $j$; see~\cite[Theorem~6.1]{MR2186989}.

\subsection{A structure theorem}\label{sect:strth}
The power of characterizing inclusions  $G\leq W$ into types as above lies in
the fact that
in many cases strong information can be deduced about the 
permutational isomorphism type of $M$.
We denote the dihedral group of
order $2a$ by $\dih a$.

\begin{theorem}\label{c2}
Assume that Hypothesis~\ref{hyp1} holds, that $M$ is non-abelian and let
$T$ be a simple factor of $M$. 
If the inclusion type is $\cd{1S}$ then let $\S$ be the direct decomposition 
of $M$ as in~\eqref{dd1s}; otherwise set $\S=\{T_1,\ldots,T_k\}$. Set
$k_1=|\S|$ and let
$\omega\in\Omega$.
\begin{enumerate}
\item $G\pi$ has at most two orbits in $\underline\ell$. Further, 
if $G\pi$ has two orbits in $\underline\ell$ then $T$ is isomorphic 
to one of the groups in 
columns
1--5 of Table~\ref{c2table}. 
\item If the inclusion $G\leq W$ has
type $\cd{1S}$ or $\cd{2\sim}$, then $T$ is 
isomorphic to 
one of the groups in columns 1--4 of Table~\ref{c2table}.
\item If the type of the inclusion $G\leq W$ is $\cd 3$, then 
$T$ is isomorphic to one of the groups
in columns 6--8 of Table~\ref{c2table}.
\item In all parts (1)--(3), we have that $M_\omega$ contains an 
$\S$-subgroup isomorphic to $X^{k_1}$ and is contained in an
$\S$-subgroup isomorphic to $Y^{k_1}$ where $X$ and $Y$ are 
as in the corresponding column of Table~\ref{c2table}. 
In particular, when $X=Y$, then $M_\omega$ is an $\S$-subgroup.
\item If the inclusion $G\leq W$ has type $\cd{2\not\sim}$, then
  $T$ can be factorized as $T=AB$ with proper subgroups $A,\ B\leq T$
  and $M_\omega$ is isomorphic to a subgroup of $(A\cap B)^k$.
\end{enumerate}
\end{theorem}

{\footnotesize \begin{center}
\begin{table}
$$
\begin{array}{|c|c|c|c|c|c|c|c|c|}
\hline
& 1 & 2 & 3 & 4 & 5 & 6 & 7 & 8\\
\hline
T & \alt 6 & \begin{array}{c}\sp 4{2^a}\\a\geq 2\end{array} &\pomegap 8q & \mat{12} & \mat{12} &
 \begin{array}{c}\sp {4a}2\\
a\geq 2\end{array} & \pomegap 83 & \sp 62\\
\hline
X& \dih{10} & \dih{2^a+1} & \mbox{\sf G}_2(q) & \psl 2{11} &
C_{11}\rtimes C_5 & C_2\times\sp{2a-2}4 & C_6\rtimes \sy 4 & \dih 3\\
\hline
Y& \dih{10} & \dih{2^a+1}\cdot 2 & \mbox{\sf G}_2(q) & \psl 2{11} &
C_{11}\rtimes C_5 & C_2\times\sp{2a-2}4 & C_6\rtimes \sy 4 & \dih 3\cdot 2\\
\hline
\end{array}
$$
\caption{Table for Theorem~\ref{c2}}
\label{c2table}
\end{table}
\end{center}}



\begin{proof}
Using the observations on the cartesian system of subgroups
in the proof of Theorem~\ref{notmax}, the claims of this theorem
are simple consequences of the results published in~\cite{MR2004334,MR2186989,MR2342458,MR2346470}.
(1) follows from Theorems~8.1 and~8.2  of~\cite{MR2346470}.
(2)--(3) follow from~\cite[Theorem~6.3]{MR2186989}.
(4) follows from Theorem~6 of~\cite{MR2004334},
Proposition~5.2, Theorem~7.1 and Proposition~7.2
of~\cite{MR2342458}. Claim~(5) follows from Lemma~\ref{notmax}.
\end{proof}

When $M=T$ is non-abelian simple, then an inclusion
$G\leq W$ as in~Hypothesis~\ref{hyp1} such that $G\pi$ is transitive must  have type $\cd{2\sim}$
(see~\cite[Theorem~6.2]{MR2186989}). It is an important consequence 
of Theorem~\ref{constrnormal} that in all 
of the cases of Theorem~\ref{c2} when $M_\omega$ is an $\S$-subgroup with 
$k_1\geq 2$, there is also a normal inclusion $G\leq\sym\Xi\wr\sy{k_1}$.





\section{Cartesian decompositions and arc-transitive graphs}\label{cdgraphs}

We will work under the following assumptions.
\begin{hypothesis}\label{hyp2}
Assume Hypothesis~\ref{hyp1} and, in addition, that 
$\Gamma$ is a connected graph with vertex set $\Omega$ and that
$G$ is a subgroup of  $\aut\Gamma$.
\end{hypothesis}

By Lemma~\ref{minbase},
$M\leq(\Sym\Delta)^\ell$. 
We denote by $\ver\Gamma$ and $\ed\Gamma$ the vertex set 
and the edge set of $\Gamma$,
respectively.
For $\alpha\in\ver\Gamma$, we let $\Gamma(\alpha)$ denote the neighborhood 
$\{\beta\mid\{\alpha,\beta\}\in\ed\Gamma\}$ of $\alpha$ in $\Gamma$.
Then the stabilizer $G_\alpha$ acts on $\Gamma(\alpha)$. 
As $G$ is assumed to be transitive on $\ver\Gamma$, the graph $\Gamma$ 
is $(G,2)$-arc-transitive if and only
if $G_\alpha$ is 2-transitive on $\Gamma(\alpha)$.

\begin{lemma}\label{marctrans}
Suppose that Hypothesis~\ref{hyp2} holds, that 
$M$ is non-regular and that $G_\alpha$ is quasiprimitive on 
$\Gamma(\alpha)$. 
Then
$\Gamma$ is $M$-arc-transitive, and so $M_\alpha$ is transitive 
on $\Gamma(\alpha)$ for all $\alpha\in\ver\Gamma$. 
\end{lemma}
\begin{proof}
As $M_\alpha\unlhd G_\alpha$,
either $M_\alpha$ is transitive or $M_\alpha$ is trivial on 
$\Gamma(\alpha)$. Suppose that
$M_\alpha$ is trivial. We show that in this case $M_\alpha$ fixes
all vertices of $\Gamma$, which is impossible as $M$ is not regular. 
We have by the conditions that $M_\alpha$ fixes all vertices in 
$\{\alpha\}\cup\Gamma(\alpha)$; that is, it fixes all vertices of distance 
less than or equal to $1$ from $\alpha$. Thus if $\beta\in\Gamma(\alpha)$, 
then $M_\beta=M_\alpha$ and $M_\beta$ fixes all vertices of 
$\Gamma(\beta)$. Thus $M_\alpha$ fixes all vertices with distance at most 2
from $\alpha$. Continuing by induction and using that $\Gamma$ is connected, 
$M_\alpha$ fixes all vertices of $\Gamma$. This is a contradiction, and hence
$M_\alpha$ is transitive on $\Gamma(\alpha)$. 
\end{proof}

Suppose that $\pi:W\rightarrow\sy\ell$ is the natural projection, and,
for $j\in\underline\ell$, let $\comp Gj$ be the component as
defined in Section~\ref{sect:cs}.
If $G\pi$ is transitive, then we may assume, by Theorem~\ref{comps}, 
that $G\leq\comp G1\wr(G\pi)$.

\begin{lemma}\label{lemmabeta}
Assume Hypothesis~\ref{hyp2}, that $\Gamma$ is $M$-arc-transitive, 
that $G\pi$ is transitive on $\underline\ell$ and that $G\leq\comp G1\wr(G\pi)$.
Let $\alpha=(\alpha_1,\ldots,\alpha_1)\in\Omega$,  
$\beta=(\beta_1,\ldots,\beta_\ell)\in\Gamma(\alpha)$ and set
$K=\comp G1$. Then
there is some $h\in (K_{\alpha_1})^\ell$ such that $\beta h=(\beta_1,\ldots,\beta_1
)$.
Thus, in the image graph $\Gamma h$, the tuples $(\alpha_1,\ldots,\alpha_1)$
and $(\beta_1,\ldots,\beta_1)$ are connected. Furthermore, 
$G^h\leq\aut{(\Gamma h)}$ with 
$G^h\leq \comp G1\wr (G\pi)$. 
\end{lemma}
\begin{proof}
Since $M_\alpha$ is transitive on $\Gamma(\alpha)$ and 
$M_\alpha\leq (K_{\alpha_1})^\ell$, we obtain that
$$
\Gamma(\alpha)\subseteq \beta_1K_{\alpha_1}\times\cdots\times 
\beta_\ell K_{\alpha_1}
$$
where $\beta_i K_{\alpha_1}$ is the $K_{\alpha_1}$-orbit containing
$\beta_i$ and the product is the cartesian product of sets.
We claim that $\beta_1K_{\alpha_1}=\beta_iK_{\alpha_1}$ 
for all $i$.  Choose $i\in\{1,\ldots,\ell\}$. 
Since $M$ is transitive on $\Omega$, $G=MG_\alpha$. 
Since $G\pi$ is transitive on $\ell$ and, by Lemma~\ref{minbase}, 
$M$ is in the kernel of
$\pi$, $G_\alpha\pi$ must be transitive on $\underline\ell$.  
Thus there is $g\in G_\alpha$ such that $g=(h_1,\ldots,h_\ell)\sigma$ 
with $h_1,\ldots,h_\ell\in K_{\alpha_1}$, $\sigma\in G\pi$
and
$1\sigma=i$.
Now 
$$
\beta g=(\beta_1,\ldots,\beta_\ell)(h_1,\ldots,h_\ell)\sigma=(\beta_1h_1,\ldots,\beta_\ell h_\ell)\sigma.
$$
Thus the $i$-th entry of $\beta g$ is $\beta_1h_1$ and so the $i$-th entry
of $\beta g$ is contained in $\beta_1 K_{\alpha_1}$. However, the $i$-th entry of 
$\beta g$ is also contained in $\beta_i K_{\alpha_1}$ which gives that $\beta_iK_{\alpha_1}=\beta_1
K_{\alpha_1}$. 
Thus there is an element 
$h=(h_1,\ldots,h_\ell)\in (K_{\alpha_1})^\ell$ such that $\beta h=(\beta_1,\ldots,\beta_1)$. Clearly, $G^h\leq \aut{(\Gamma h)}$; further
$G^h\leq (K\wr (G\pi))^h=K\wr (G\pi)$ since $h\in K^\ell$. 
\end{proof}

If $\Gamma_1$ and $\Gamma_2$ are graphs, then the {\em direct product} of 
$\Gamma_1\times\Gamma_2$ is defined as the graph whose vertex set
is the cartesian product $\ver{\Gamma_1}\times\ver{\Gamma_2}$ and adjacency is 
defined by the rule that $(\gamma_1,\gamma_2),\ (\gamma_1',\gamma_2')$ 
are adjacent in $\Gamma_1\times\Gamma_2$ if and only if $\gamma_1$ is adjacent 
to $\gamma_1'$ in $\Gamma_1$ and 
$\gamma_2$ is adjacent to $\gamma_2'$ in $\Gamma_2$. The definition 
of direct product can be extended to 
an arbitrary finite number of factors, and in particular, if $\Gamma_1$ 
is a graph and $l\geq 2$, then one can define $(\Gamma_1)^l$ as the graph with
vertex set $\ver{\Gamma_1}^l$ such that $(\gamma_1,\ldots,\gamma_l)$ 
is adjacent to $(\gamma_1',\ldots,\gamma_l')$ in $(\Gamma_1)^l$ 
if and only if $\gamma_i$ is adjacent to $\gamma_i'$ in $\Gamma_1$ for all $i$. Note that if
$\alpha_1\in\Gamma_1$ and $\alpha=(\alpha_1,\ldots,\alpha_1)$ then
the neighborhood $(\Gamma_1)^l(\alpha)$ of $\alpha$ in $(\Gamma_1)^l$ is 
equal to the cartesian power $(\Gamma_1(\alpha_1))^l$ 
of the neighborhood of $\alpha_1$ in $\Gamma_1$. See~\cite{productgr1}
for more background on products of graphs and their properties.

\begin{proposition}\label{grnormal}
Suppose that Hypothesis~\ref{hyp2} holds, that $\Gamma$ is $M$-arc-transitive 
and that the inclusion of $G$ into $W$ is normal. Then 
$\Gamma\cong(\Gamma_1)^\ell$ where $\Gamma_1$ is the graph whose vertex set
is $\Delta$ and edge set is the $\comp M1$-orbit of 
$\{\alpha_1,\beta_1\}$ where $\alpha_1,\beta_1\in\Delta$ are chosen so that
$(\alpha_1,\ldots,\alpha_1)$ and $(\beta_1,\beta_2,\ldots,\beta_\ell)$ 
are connected in $\Gamma$. In particular, $\Gamma$ is not 
$(G,2)$-arc-transitive.
\end{proposition}
\begin{proof}
Since the inclusion of $G$ into $W$ is normal, we find that $G\pi$ is 
transitive (Proposition~\ref{normalstab}),
and hence we may assume without loss of generality by Theorem~\ref{comps} that 
$G\leq\comp G1\wr(G\pi)$ and, by Lemma~\ref{lemmabeta}, that
there are $\alpha_1,\beta_1\in \Delta$ such that 
$\alpha=(\alpha_1,\ldots,\alpha_1)$ and $\beta=(\beta_1,\ldots,\beta_1)$ 
are adjacent in $\Gamma$. Since $M$ is arc-transitive, we find that
$M_\alpha$ is transitive in $\Gamma(\alpha)$. Further, 
by Proposition~\ref{normalstab}, $M_\alpha=\prod_j(\comp Mj)_{\alpha_1}$. 
As $\Gamma$ is $M$-arc-transitive, $M$ is not regular on $\Omega$. Thus
$\comp Mj$ is the unique minimal normal subgroup of $\comp Gj$. Since 
$\comp Gj=\comp G1$, we have that $\comp Mj=\comp M1$ for all $j$.
Thus
$$
\Gamma(\alpha)=\prod_{i=1}^\ell \beta_1 (\comp M1)_{\alpha_1}.
$$
If $\Gamma_1$ denotes the graph whose vertex set is $\Delta$ 
and whose edge set is the $\comp M1$-orbit containing $\{\alpha_1,\beta_1\}$
then, by, the remark before the proposition, we have that the neighborhoods
of $\alpha$ in $\Gamma$ and in $(\Gamma_1)^\ell$ are equal. Since $M$ 
acts vertex transitively on both $\Gamma$ and $(\Gamma_1)^\ell$, we find that
$\Gamma=(\Gamma_1)^\ell$.
As the valency of $\Gamma$ is at least 2, we obtain that
$\Gamma_1$ also has valency at least 2.
Hence if $\beta_2\in\ \beta_1 (\comp M1)_{\alpha_1}\setminus\{\beta_1\}$, then
both $(\beta_1,\ldots,\beta_1,\beta_2)$ and 
$(\beta_1,\ldots,\beta_1,\beta_2,\beta_2)$ are in 
$\Gamma(\alpha)\setminus\{\beta\}$, however no element of 
$K_{\alpha_1,\beta_1}\wr \sy\ell$ can map one of these points to the other.
Hence $G_\alpha$ is not 2-transitive on $\Gamma(\alpha)$ and so
$\Gamma$ is not $(G,2)$-arc-transitive.
\end{proof}




\begin{corollary}\label{nonnormcor}
Suppose that Hypothesis~\ref{hyp2} holds, that $M$ is not regular, and 
that the inclusion $G\leq W$ is normal. 
Then $\Gamma$ is not $(G,2)$-arc-transitive.
\end{corollary}
\begin{proof}
Let $\alpha$ be a vertex of $\Gamma$. 
If $\Gamma$ is $(G,2)$-arc-transitive, then $G_\alpha$ is $2$-transitive on
$\Gamma(\alpha)$ and hence $\Gamma$ is $M$-arc-transitive; by Proposition~\ref{grnormal}, this is
a contradiction. Thus $\Gamma$ is not $(G,2)$-arc-transitive.
\end{proof}

We note that if $G$ is a quasiprimitive group of type {\sc Cd}, then
its unique minimal normal subgroup $M$ is  not regular and $M_\omega$ is 
an $\S$-subgroup with respect to the direct product decomposition of $M$
determined by the simple factors of $M_\omega$.
Hence By Corollary~\ref{nonnormcor}, for such a $G$, there exist
no $(G,2)$-arc-transitive graphs. This argument gives an alternative proof
of~\cite[Lemma~5.3(i)]{prae:quasi}.

Finally in this section we  exclude the case of abelian $M$.

\begin{proposition}\label{lem6}
Under Hypothesis~\ref{hyp2}, if $\Gamma$ is $(G,2)$-arc-transitive, then $M$ is non-abelian.
\end{proposition}
\begin{proof}
Suppose that $M$ is abelian. 
Then $G$ is an affine primitive
permutation group and, for $\alpha\in\Omega$,
 $G=M\rtimes G_\al$ where $M$ is viewed as a $k$-dimensional vector space 
over $\mathbb F_p$  (see~\cite[Section~3]{bam-prae}).
Hence in this proof
we use additive notation in $M$.
Further, $M$, viewed as an $G_\alpha$-module, is irreducible.
By Lemma~\ref{minbase}, $M\leq(\sym\Delta)^\ell$ and so $M\leq \comp M1\oplus\cdots\oplus\comp M\ell$.
As each component $\comp Mj$ is transitive, and hence  regular,
on $\Delta$, we find that $|\comp Mj|=|\Delta|$. Since
$|M|=|\Omega|=|\Delta|^\ell=\prod_j|\comp M j|$,
we obtain that $M=\comp M1\oplus\cdots\oplus\comp M\ell$
and the components $\comp Mj$
are permuted transitively by $G_\alpha$.  In other words, $G_\alpha$
preserves a direct sum decomposition of $M$ into $\ell$ components;
that is, $G_\alpha$ is primitive as a linear group acting on $M$.

The class of $(G,2)$-arc-transitive
graphs with primitive affine  groups $G$ were classified
by~\cite{iva-prae}. We need to consider those lines in~\cite[Table~1]{iva-prae}
in which the corresponding group is primitive; that is,
the entry in the ``Prim'' column is ``P''.
We will show for these lines that a point stabilizer
$G_\al$ does not preserve a non-trivial direct sum decomposition of $M$.

In Lines 1--2, $G_\alpha$ is  $\gl k2$ or $\sy{k+1}$
and these groups do not preserve non-trivial
direct sum decompositions.
In Line 7, $G_\alpha=\pgammal m{2^a}$ acting in dimension $k=m^a$ with $m\geq 3$
and $a\geq 1$.
By the discussion in~\cite[(1.3)]{iva-prae}, $M$ is an irreducible
module for $\psl m{2^a}$. We claim that $M$ is a primitive  
$\psl m{2^a}$-module. If $M$ were imprimitive, then $\psl m{2^a}$ would permute
 non-trivially
an imprimitivity decomposition $M=V_1\oplus\cdots\oplus V_\ell$ with 
$\ell\leq m^a$. On the other hand, if $(m,2^a)\neq (4,2)$, 
the minimal degree permutation representation
of $\psl m{2^a}$ is on the 1-spaces of the natural module and has
degree $(2^{ma}-1)/(2^a-1)>2^{(m-1)a}$, while the minimal degree
of a permutation representation of $\SL 42$ is 8; 
see~\cite[Table~1]{cooperstein}. 
As $m\geq 3$, we have that $m<2^{m-1}$, and so $m^a<2^{(m-1)a}$. Further,   
$4^1<8$, and so we obtain that $\psl m{2^a}$ cannot
be imprimitive in these cases.

We have, in the remaining two lines,
that $G_\alpha\cong M_{23}$ or $G_\alpha\cong M_{22}.2$; these lines
can be be handled similarly.
If the underlying module of $M_{23}$ were imprimitive, 
then $M_{23}$ would be acting on 11 points transitively, which is impossible. 
If the underlying module for $M_{22}.2$ were imprimitive, then either
$M_{22}$ would be acting in dimension less then $10$ or it would be 
acting transitively on at most 10 points; both of these possibilities 
are clearly impossible by~\cite{Atlas}.
\end{proof}

\section{Simple plinth}\label{sect:simpleplinth}

In this section we will consider the situation
described in~Hypothesis~\ref{hyp2} under the additional condition that
$M=T$ is a non-abelian simple group. Unlike in the other cases, here 
we  find two examples. Further, we show that these examples
are the only possibilities.

\begin{theorem}\label{simpleplinth}
Suppose that Hypothesis~\ref{hyp2} holds. Assume, further that $M=T$ is a
non-abelian simple group and that $\Gamma$ is $(G,2)$-arc-transitive.
Then 
one of the following is valid.
\begin{enumerate}
\item $T=\alt 6$, $|\aut\alt 6:G|\in\{1,2\}$,
  $G\neq\pgl 29$, and $|\ver \Gamma|=36$, $\Gamma$ is 
Sylvester's Double Six Graph
of valency $5$.
\item $T=\sp 44$, $G=\aut\sp 44$, $|\ver\Gamma|=14,400=120^2$, and $\Gamma$ is a graph of valency $17$.
\end{enumerate}
In both cases, $G$ is quasiprimitive.
\end{theorem}
\begin{proof}
Let $\alpha\in\ver\Gamma$. 
By~\cite[Theorem~6.1]{cs}, 
$T_\alpha$ is non-trivial.
As $\Gamma$ is connected and $G_\alpha$ is $2$-transitive,
and in particular quasiprimitive,
on $\Gamma(\alpha)$,
Lemma~\ref{marctrans} gives that
$\Gamma$ is $T$-arc-transitive. In particular $\Gamma(\alpha)$ is a 
$T_\alpha$-orbit. As $T$ is simple, 
the possibilities for $T$ and $T_\alpha$ are given 
in~\cite[Theorem~6.1]{cs}. In particular, $T$ is isomorphic to
one of the groups $\alt 6$, $\mat{12}$, $\sp 4{2^a}$, with $a\geq 2$, or
$\pomegap 8q$. We analyze each of these cases below.

$T=\alt 6$, $|\ver\Gamma|=36$: In this case  calculation with GAP~\cite{GAP}
or Magma~\cite{Magma}
can show that $G$ has a suborbit of size $5$ on which $G_\alpha$ acts 
2-transitively if and only if $\sy 6\leq G$. 
The corresponding 2-arc-transitive
graph is undirected, connected  and is of valency 5 as described in item~(1).

$T=\mat{12}$ and $|\ver\Gamma|=144$: Calculation with GAP~\cite{GAP}
or Magma~\cite{Magma} shows that $T$ has no suborbit on which
$G_\alpha$ is 2-transitive and the corresponding
graph is connected. Hence no graph arises in this case.

$T=\sp{4}{q}$ with $q=2^a$ and $a\geq 2$: The group  $G$ is contained in
$\aut T\cong \sp 4q.C_a.C_2$ where $C_a$ is the group of field automorphisms
and $C_2$ is generated by a graph automorphism.
Further, $T_\alpha=Z\left<\sigma\right>$ 
where $Z$ is a subgroup
of order $q^2+1$ of a Singer cycle and $\sigma$ is an element of order 4; 
see~\cite[Lemma~5.2]{cs}. Let $\gamma\in\Gamma(\alpha)$ and let $x\in T$ such that
$\gamma=\alpha x$. 

We claim that 
$Z\cap Z^x=1$. Suppose that $1\neq Z_1=Z\cap Z^x$. Then 
$Z_1$ is cyclic of odd order and suppose that $R$ is a cyclic subgroup of $Z_1$
with odd prime order $r$. Then $r\mid q^2+1$, and so $r\nmid q^2-1$, and hence 
$r\nmid q-1$ and $r\nmid q^3-1=(q-1)(q^2+q+1)$. Thus $r$ is a primitive
prime divisor of $q^4-1$ and hence the centralizer of $R$ in $\gl 4q$ is a full 
Singer cycle $C_{q^4-1}$~\cite[7.3 Satz]{huppert}. 
Hence $Z$ and $Z^x$ are subgroups of the same
$C_{q^4-1}$, which 
implies that $Z=Z^x$. Since $T_\alpha=\norm {T}{Z}$ and $T_\gamma=\norm T{Z^x}$ we find that $T_\alpha=T_\gamma$.
As $T_\alpha$ is 
self-normalizing in $T$, we obtain that $\alpha=\gamma$ which is 
a contradiction as $\gamma\in\Gamma(\alpha)$ is assumed. 
This shows that $Z\cap Z^x=1$ as claimed. 

As $Z\cap Z^x=1$,  $Z$ is faithful on
$\Gamma(\alpha)$. As $Z$ is normal in $G_\alpha$ and $G_\alpha$ is 2-transitive
on $\Gamma(\alpha)$, $Z$ is transitive, and hence regular, on $\Gamma(\alpha)$. 
Thus $|\Gamma(\alpha)|=q^2+1$. As  $|G_\alpha|\leq 8a(q^2+1)$,
and $G_\alpha$ is 2-transitive on $\Gamma(\alpha)$, it follows that
$a=2$. If $a=2$, then  there exists a corresponding 2-arc-transitive
graph $\Gamma$, which is given in item~(2).

$T=\pomegap 8q$: In this case $T_{\alpha}=G_2(q)$. If $q\geq 3$, then
$G_2(q)$ does not have a 2-transitive permutation representation. As is 
well-known, $G_2(2)\cong \psu 33.2$, and $\psu 33.2$ has a 2-transitive 
representation of degree $q^3+1=28$. However, computation with 
GAP~\cite{GAP} or Magma~\cite{Magma} shows that $T_\alpha$
has no suborbit of size $28$, and hence this case does not arise.
\end{proof}

Computation with GAP~\cite{GAP} or Magma~\cite{Magma} 
can verify that the 2-arc-transitive graphs with $T=\alt 6$ and $T=\sp 44$
do in fact exist. The graph with $T=\alt 6$ is also known as Sylvester's Double 
Six Graph, it has 36 vertices and valency 5 
(see~\cite[13.1.2~Theorem]{distancebook}). Its full automorphism group
is $\aut\alt 6\cong\pgammal 29$. The graph with $T=\sp 44$ has $120^2=14,400$ 
vertices and valency 17. Both graphs that appear in this context can be 
described using the generalized quadrangle associated with the 
non-degenerate alternating form stabilized by $\sp 4q$, but giving such a 
description is beyond the scope of the present paper.

\section{Non-simple plinth}\label{sec:nonsimple}

We continue working under Hypothesis~\ref{hyp2}. 
We know from Proposition~\ref{lem6} that if
$\Gamma$ is $(G,2)$-arc-transitive, then
$M$ is non-abelian. Further, the situation when $\Gamma$ is $(G,2)$-arc-transitive and $M$ is simple is fully
described in Section~\ref{sect:simpleplinth}.
Hence from now of we will focus on the case when $M$ is non-abelian
and non-simple.
We will also assume that $M$ is not regular, as for the case
when $M$ is regular, the graph will be a Cayley graph; see~\cite{baddeleytw1}.

We note that, using~\cite[Theorem~1.3]{liseresssong},
our analysis could by somewhat simplified if we assumed that $G$ 
was quasiprimitive. However, we opt for keeping our general conditions set
up in~Hypothesis~\ref{hyp2}.

\begin{theorem}\label{nonsimple}
Suppose that Hypothesis~\ref{hyp2} is valid, that $\Gamma$ is $(G,2)$-arc-transitive
and that $M$ is non-regular and non-simple.
Then $G\pi$ is transitive and 
the inclusion $G\leq W$ is of type $\cd{2\not\sim}$.
\end{theorem}
\begin{proof}
By Corollary~\ref{nonnormcor}, the inclusion $G\leq W$ is not normal. 
We need to show that the cases when $G\pi$ is intransitive, or 
the inclusion $G\leq W$ has type  $\cd{1S}$, $\cd{2\sim}$, or $\cd 3$
are impossible.
Assume, by contradiction, that one of these cases is valid
and Theorem~\ref{c2} applies. Let $\alpha$ be a vertex of $\Gamma$. 
Whenever, in 
Theorem~\ref{c2}, the stabilizer $M_\alpha$ is an $\S$-subgroup
with $k_1\geq 2$,
there exists, by Theorem~\ref{constrnormal}, a normal embedding $G\leq \sym\Xi\wr\sy{k_1}$.
If $k_1=k$, then $k_1\geq 2$ by the conditions of the theorem.
If the inclusion type is $\cd{1S}$, then, by the definition 
of this inclusion type, a component 
$\comp M1$ contains a full diagonal subgroup in $T_{i_1}\times T_{i_2}$, 
but there is also a simple factor $T$ of $M$ such that the projection
$(\comp M1)_\delta\overline T/\overline T$ is a proper subgroup of $T$.
Hence this $T$ is different from $T_{i_1}$ and $T_{i_2}$ and this implies
that $k_1\geq 2$ also in this case.
Hence in these cases, we obtain the desired
contradiction by Corollary~\ref{nonnormcor}.
Thus we are left with the possibilities of Theorem~\ref{c2}(1)--(4)
in which $M_\alpha$ is not an $\S$-subgroup. In these
cases $T\cong \sp nq$ where the possibilities for the inclusion type, $n$, 
$q$, and $k_1$ are summarized in Table~\ref{exceptions}.

\begin{center}
\begin{table}
$$
\begin{array}{|c|c|c|c|}
\hline 
\begin{array}{c}
\mbox{inclusion}\\\mbox{type}\end{array}& n & q & k_1  \\
\hline
\mbox{intransitive}
& 4 & 2^a\geq 4 & k \\
\hline
\cd{2\sim} & 4 & 2^a\geq 4 & k \\
\hline
\cd{3}  &6 & 2 & k\geq 3\\
\hline
\cd{1S} & 4 & 2^a\geq 4 & k/2\geq 2 \\
\hline
\end{array}
$$
\caption{}
\label{exceptions}
\end{table}
\end{center}

In each of the cases in Table~\ref{exceptions},
we have that
\begin{equation}\label{ineq}
  (\dih {q+1})^{k_1}\leq M_\alpha\leq (\dih {q+1}\cdot 2)^{k_1}.
  \end{equation}
Let $A$ and $B$ denote the subgroups of the left and the right hand side,
respectively, of~\eqref{ineq}.
Hence in each case we have that $A\leq M_\alpha\leq B$. 
Let $C$ denote the common Hall $2'$-subgroup of $A$ and $B$. 
Thus
$C$ is isomorphic to  $(C_{q+1})^{k_1}$.
By the conditions, $C$  is a subgroup of $M_\alpha$
and  $B=(\dih {q+1}\cdot 2)^{k_1}$.
For $i=1,\ldots,k_1$, we let $\sigma_i$ denote
the $i$-th coordinate projection $B\rightarrow \dih{q+1}\cdot 2$. 

\begin{claim}\label{dihnorms}
If $N$ is a normal subgroup of $G_\alpha$ contained in $C$, then
$N=(C_s)^{k_1}$ where $s$ is a divisor of ${q+1}$. Conversely, each such 
subgroup of $C$ is normal in $G_\alpha$.
\end{claim}
\begin{proof}
We have that $N\sigma_i\leq C_{q+1}$. The subgroup $C$, being characteristic
in $M_\alpha$, is normalized by $G_\alpha$, and
$G_\alpha$ preserves the direct decomposition $C=(C_{q+1})^{k_1}$.
Thus $N\sigma_i$ is independent
of $i$. Therefore $N\sigma_i$ is a
cyclic 
subgroup of order $s$, say,
of $C_{{q+1}}$, and hence $N$ is a subdirect subgroup of 
$(C_s)^{k_1}$. Suppose that $x_1\in C_s$. Then 
there is an element of the form $(x_1,\ldots,x_{k_1})$ in $N$ with $x_i\in C_s$.
Suppose that $x\in \dih{{q+1}}$ is an
involution that conjugates each element of $C_{q+1}$ to its inverse.
Since $M_\alpha$ contains $(\dih{{q+1}})^{k_1}$, 
$(1,x,\ldots,x)\in M_\alpha$,
and hence 
$$
(x_1,x_2,\ldots,x_{k_1})^{(1,x,\ldots,x)}=(x_1,x_2^{-1},\ldots,x_{k_1}^{-1})\in N,
$$
which implies that 
$$
(x_1,x_2,\ldots,x_{k_1})(x_1,x_2^{-1},\ldots,x_{k_1}^{-1})=(x_1^2,1,\ldots,1)\in N.
$$
Since $s$ is odd, the order $|x_1|$ of $x_1$ is coprime to $2$, which
gives that $(x_1,1,\ldots,1)\in N$.
This shows that $N$ contains $N\sigma_1$. The same argument shows that 
$N$ contains $N\sigma_i$ for all $i$, and hence $N=(C_s)^{k_1}$ as claimed.

A
subgroup of the form $(C_s)^{k_1}$ is normalized by $G_\alpha$, since 
 $(C_{{q+1}})^{k_1}$ is normalized by $G_\alpha$ (it is the unique
$2'$-subgroup
of $M_\alpha$) and $(C_s)^{k_1}$ is the largest subgroup of $(C_{{q+1}})^{k_1}$ 
with exponent $s$, and hence it is characteristic in $(C_{{q+1}})^{k_1}$ and 
normal in $G_\alpha$.
\end{proof}



Suppose that $\Gamma$ is a $(G,2)$-arc-transitive graph, as assumed in the
theorem.
We know that $Q=G_\alpha^{\Gamma(\alpha)}$ is a 2-transitive group. 
In particular $G_\alpha$ is primitive on $\Gamma(\alpha)$ and so $M_\alpha$ is
either transitive on $\Gamma(\alpha)$ or trivial. By Lemma~\ref{marctrans}, we obtain
that $M_\alpha$ is transitive.
Since $M_\alpha\unlhd G_\alpha$ and $M_\alpha$ is solvable,
we find that $Q$ is an affine 2-transitive group. 
In particular the socle $V$ of $Q$ is $V=\Z_p^d$ with some prime $p$ and 
$Q=V\rtimes Q_\beta=\Z_p^d\rtimes G_{\alpha,\beta}^{\Gamma(\alpha)}$. 

A linear group $H\leq\gl dp$ is said to be {\em monomial} if $H$ preserves 
a direct sum decomposition $V_1\oplus\cdots\oplus V_d$ of $V=(\mathbb F_p)^d$
with subspaces of dimension 1. If $H$ is monomial, then $H$ cannot be 
transitive on $V\setminus\{0\}$, and hence $V\rtimes H$ cannot
be 2-transitive on $V$. In particular, the linear group $Q_\beta$ above is not
monomial.

Set $\kr G\alpha$ to be the kernel of $G_\alpha$ in its action on 
$\Gamma(\alpha)$. 
Define $\kr M\alpha$
accordingly and note that $\kr M\alpha$ is normal in $G_\alpha$.
Recall that $C=(C_{q+1})^{k_1}$ as defined above and that $C$ is normal in $G_\alpha$.

\begin{claim}\label{claim2}
$C\leq \kr M\alpha$.
\end{claim}
\begin{proof}
Let $C_1=C\cap \kr M\alpha=C\cap \kr G\alpha$. 
Then $C_1$ is normal in $G_\alpha$. Hence 
$C_1=(C_s)^{k_1}$ with some $s|{q+1}$ by Claim~\ref{dihnorms}. 
Thus $C/C_1=C/(C\cap \kr G\alpha)=
C\kr G\alpha/\kr G\alpha$ can be considered as a normal subgroup of
$G_\alpha^{\Gamma(\alpha)}$. If $C_1\neq C$, then this is
a non-trivial normal subgroup of $G_\alpha^{\Gamma(\alpha)}$ and hence it contains
a minimal normal subgroup of the form $V=(C_r)^{k_1}$ with some prime
divisor $r$ of $q+1$. On the other hand, since
$G_\alpha$ permutes the $k_1$ factors of $C$, $G_\alpha$ induces a monomial
subgroup on $(C_r)^k$, and hence $G_\alpha$ cannot be transitive
on the set of non-trivial elements of $V$ (see the remark above). 
This is a contradiction, and 
hence $C_1=C$, which means that $C\leq \kr M\alpha$.
\end{proof}

Let $\beta\in\Gamma(\alpha)$. By Claim~\ref{claim2}, 
 $C\leq M_\beta$. Since $\kr M\beta$ is $G$-conjugate to $\kr M\alpha$, 
it contains the unique Hall $2'$-subgroup of $M_\beta$. 
On the other hand, this Hall
$2'$-subgroup is equal to $C$, and hence $C\leq \kr M\beta$, and
so $C\leq M_\delta$ for all $\delta\in\Gamma(\beta)$. 
Then the same argument implies that if $\delta\in\Gamma(\beta)$, then
$C\leq \kr M\delta$. Since $\Gamma$ is assumed to be connected, this implies
that $C$ acts trivially on the vertex set $\Omega$ of $\Gamma$, which is
impossible, as $G$ is assumed to be a subgroup of $\sym\Omega$.

This shows that the assumption that either $G\pi$ is intransitive
or that the inclusion $G\leq W$ is one of the types $\cd {1S}$, $\cd{2\sim}$,
or $\cd 3$ always leads to a contradiction. Hence the inclusion $G\leq W$ has
to have type $\cd {2\not\sim}$. 
\end{proof}

Now Theorem~\ref{main} is a consequence of 
Proposition~\ref{lem6}, Corollary~\ref{nonnormcor}, 
Theorem~\ref{simpleplinth} and Theorem~\ref{nonsimple}.

\section{The Li--Seress examples}\label{sect:ls}

By Theorem~\ref{main}, if $\Gamma$ is a connected 
$(G,2)$-arc-transitive
graph for an innately transitive permutation group $G$ with a 
non-regular, non-simple plinth, then
an inclusion $G\leq W$ into a wreath product $W$ in product action 
as in Hypothesis~\ref{hyp1} has to have
type $\cd{2\not\sim}$. 
On the other hand, we 
do not know if such $(G,2)$-arc-transitive graphs exist.
The information about the finite 
simple factor $T$ of the plinth and its factorization given in
Lemma~\ref{notmax} is not as strong in this case as in the 
cases described by Theorem~\ref{c2}.
Thus  deciding if such examples exist seems to be a difficult problem. 
By Theorem~\ref{main}, 
the O'Nan--Scott type of a quasiprimitive group $G$ in an example
must be {\sc Pa}. 
We show in this section that the 
examples of $(G,2)$-arc-transitive graphs
with quasiprimitive groups of type {\sc Pa}
in~\cite{MR2258005} do not admit an inclusion $G\leq W$ 
of type $\cd{2\not\sim}$. 

In all the examples of $(G,2)$-arc-transitive graphs in~\cite{MR2258005}, 
$G$ is a quasiprimitive group with socle $M=T^k$ where $T\cong\psl 2q$ 
with some $q$. Thus to rule out inclusions of type $\cd{2\not\sim}$ 
for these groups, we need to review some known facts about
factorizations of $\psl 2q$. The following lemma is a 
consequence of~\cite[Tables~1 and 3]{lps:max} 
and Dickson's description of the maximal subgroups of 
$\psl 2q$ (see~\cite{giudici} for an easy reference).

\begin{lemma}\label{factpsl}
Suppose that $T=\psl 2q$ with some prime power $q\geq 4$.
\begin{enumerate}
\item If $q\equiv 1\pmod 4$ and $q\not\in\{5,9,29\}$, 
then $T$ does not have a factorization $T=AB$ with proper 
subgroups $A,\ B<T$.
\item If $q\not\equiv 1\pmod 4$, then $T$ can be factorized as
$T=AB$, where $A$ 
is a maximal parabolic subgroup and
$$
B=\left\{\begin{array}{ll}
\dih{(q+1)} & \mbox{when $q$ is even;}\\ 
\dih{(q+1)/2} & \mbox{when $q$ is odd.}\end{array}\right.  
$$
Further,
$$
|A\cap B|=\left\{\begin{array}{ll}
2 & \mbox{when $q$ is even;}\\ 
1 & \mbox{when $q$ is odd.}\end{array}\right.
$$
\item If $q\not\equiv 1\pmod 4$ and $q\not\in\{7,11,19,23,59\}$, then the
only factorization $T=AB$ of $T$ with maximal subgroups $A,\ B$ 
is as in~(2).
\item If $q\in\{5,7,9,11,19,23,29,59\}$, then 
the factorizations $\psl 2q=AB$ with maximal subgroups $A$ and $B$ are either listed in part~(2) or in 
Table~\ref{excfact}, where $P_1$ denotes a maximal parabolic 
subgroup in $T$.
\end{enumerate}
\end{lemma}

\begin{center}
\begin{table}
$$
  \begin{array}{|c|c|c|c|c|c|c|c|c|c|c|}
\hline 
q & 5 & 7 & 9 & 9 & 9 &11 & 19 & 23 & 29 & 59  \\
\hline
A & P_1 & P_1 & \sy 4 & P_1 & \alt 5 & P_1 & P_1 & P_1 & P_1 & P_1\\
\hline
B & \alt 4 & \sy 4 & \alt 5 & \alt 5 & \alt 5 & \alt 5 & \alt 5 & \sy 4 & \alt 5 & \alt 5\\
\hline
A\cap B & C_2 & C_3 & C_2\times C_2 &\sy 3 & \dih 5 & C_5 & C_3 & 1 & C_2 & 1\\
\hline
\end{array}
$$
\caption{Exceptional factorizations of $\psl 2q$}
\label{excfact}
\end{table}
\end{center}

The information about the factorizations of $\psl 2q$ in Lemma~\ref{factpsl}
allows us to characterize inclusions $G\leq W$ in wreath products in 
product action of innately transitive groups $G$ whose plinths are
isomorphic to $\psl 2q^k$.

\begin{theorem}\label{factpsl2}
Suppose that $G$ is an innately transitive group on 
$\Omega$ whose plinth $M$ is isomorphic 
to $\psl 2q^k$  with some prime power $q\geq 4$ and $k\geq 1$. Let
$T$ be a simple factor of $M$, set $\overline T=\cent MT$ and 
let $\omega\in\Omega$.
\begin{enumerate}
\item If $q\equiv 1\pmod 4$ and $q\not\in\{5,9,29\}$, then 
every inclusion of $G\leq W$ satisfying~Hypothesis~\ref{hyp1} is normal.
\item Suppose $q\equiv 3\pmod 4$ and $q\not\in\{7,11,19\}$ and assume that $G$ admits
a non-normal inclusion $G\leq W$ satisfying~Hypothesis~\ref{hyp1}. 
Then $M$ is regular and
$G$ admits a normal inclusion $G\leq \sym\Xi\wr\sy k$ where $|\Xi|=|T|$.
\item If $q$ is even and $G\leq W$ is a non-normal 
inclusion such that Hypothesis~\ref{hyp1} holds, then $M_\omega\overline T/\overline T\leq C_2$.
\item If $q\in \{5,7,9,11,19,29\}$
and $G\leq W$ is a non-normal 
inclusion satisfying~Hypothesis~\ref{hyp1}, then either $M_\omega\overline T/\overline T=1$ 
and $M$ is regular or
$M_\omega\overline T/\overline T$ 
is a non-trivial subgroup of the $A\cap B$ column of a corresponding 
column of Table~$\ref{excfact}$.
\end{enumerate}
\end{theorem}
\begin{proof}
Suppose that $G$ is an innately transitive group as in the conditions of the 
theorem and that $G\leq W$ is an inclusion satisfying Hypothesis~\ref{hyp1}.
Let $\pi:W\rightarrow\sy\ell$ be the natural projection. 
By Theorem~\ref{c2}(1),  either $G\pi$ is
transitive, or $T\cong\psl 29\cong\alt 6$.
Further, if $G\pi$ is transitive, but the inclusion is not normal, 
then, by Theorem~\ref{c2}(2)--(3), either $T\cong\psl 29\cong\alt 6$, or
the inclusion type is $\cd{2\not\sim}$. 

Let first consider the case when the inclusion type is not 
$\cd{2\not\sim}$. Hence $T\cong\psl 29\cong\alt 6$ and either 
$G\pi$ 
is intransitive, or  the inclusion is one of the types $\cd{2\sim}$
or $\cd{1S}$. Then, by Theorem~\ref{c2}(2), 
$M_\omega\overline T/\overline T\cong \dih 5$. 
Since $\dih 5$ appears in the $A\cap B$ column of Table~\ref{excfact}, 
the result holds in this case.

Let us consider now the case of $\cd{2\not\sim}$ inclusions.
In this case, $T$ is contained in exactly 
two components $\comp M{j_1}$ and 
$\comp M{j_2}$. Let $\delta\in\Delta$ and set $\omega=(\delta,\ldots,\delta)$. 
Then, by Lemma~\ref{notmax}, 
$A=(\comp M{j_1})_\delta\overline T/\overline T$  and 
$B=(\comp M{j_2})_\delta\overline T/\overline T$  are proper subgroups of 
$T$ satisfying $T=AB$ and $M_\omega\overline T/\overline T\leq A\cap B$.
If $q\equiv 1\pmod 4$ and $q\not\in\{5,9,29\}$, then,
by Lemma~\ref{factpsl}(1), $T=\psl 2q$, does 
not admit a factorization with two proper subgroups, and hence,
in this case,
inclusions of type $\cd{2\not\sim}$ are not possible. This proves 
claim~(1). If $q\equiv 3\pmod 4$, but $q\not\in\{7,11,19\}$ and
$T=AB$ with proper subgroups $A$ and $B$ of $T$, then, by 
Lemma~\ref{factpsl}(3)--(4), $A\cap B=1$. Thus, the projection 
$M_\omega\overline T/\overline T$ is trivial. Since this projection is 
independent of the choice of $T$, $M_\omega=1$ and hence $M$ is regular. Then
$M_\omega$ satisfies the conditions in Theorem~2.6 with the 
direct product decomposition of $M=T^k$ into $k$ simple factors, and hence 
$G$ can be embedded into $\sym\Xi\wr\sy k$ such that $\Xi$ is the 
right coset space of the trivial subgroup in $T$. This proves~(2).

Finally, if $q$ is even or $q\in \{5,7,9,11,19,29\}$ and $G\leq W$ is an 
inclusion of type $\cd{2\not\sim}$, then 
$M_\omega\overline T/\overline T\leq A\cap B$, where $A,\ B$ are the subgroups
defined above. If $q$ is even, then, by Lemma~\ref{factpsl}(2), 
$A\cap B\cong C_2$, and so~(3) follows.
If $q$ is one of the odd prime-powers, listed above, then
either $A\cap B=1$, and so $M$ is regular, or $A$ and $B$ must occur in Table~\ref{excfact}, and so 
claim~(4) follows also.
\end{proof}

\begin{corollary} \label{pslcor}
Suppose that Hypothesis~\ref{hyp2} holds and that $M$ is non-regular and is isomorphic to 
$\psl 2q^k$ where $k\geq 1$, $q\geq 4$ is an odd prime-power and 
$q\not\in\{5,7,9,11,19,29\}$. Then $\Gamma$ is not $(G,2)$-arc-transitive.
\end{corollary}
\begin{proof}
Suppose that $G$, $W$ and $\Gamma$ are as in Hypothesis~\ref{hyp2}. 
If $G\leq W$ is a normal inclusion, then 
$\Gamma$ is not $(G,2)$-arc-transitive by Corollary~\ref{nonnormcor}. Since $T\not\cong\psl 29\cong \alt 6$, 
Theorem~\ref{c2} implies that $G\pi$ is transitive and the type of the inclusion is $\cd{2\not\sim}$.
Thus, by Lemma~\ref{notmax}, $T=\psl 2q$ must admit a factorization $T=AB$ with proper subgroups 
$A$ and $B$. Thus the conditions of Lemma~\ref{factpsl}(3) must be valid with $q$ odd. Hence 
by Theorem~\ref{factpsl2}(2), there exists a normal inclusion $G\leq \sym\Xi\wr\sy k$. Therefore 
Proposition~\ref{grnormal} implies that $\Gamma$ is not $(G,2)$-arc-transitive.
\end{proof}

Finally, we show that the Li--Seress examples cannot be embedded into 
wreath products in product action.

\begin{corollary}\label{cor:ls}
Suppose that $\Gamma$ is one of the $(G,2)$-arc-transitive graphs in 
Examples~3.2, 3.3, 3.4, 4.1, or 5.3 in~\cite{MR2258005}. Then 
$G$ cannot be embedded into a wreath product $W$ as in~Hypothesis~\ref{hyp1}.
\end{corollary}

\begin{center}
\begin{table}
$$
\begin{array}{|c|c|c|}
\hline
\mbox{Example} & \mbox{conditions for $q$} & M_\omega \overline T/\overline T \\
\hline
\mbox{Example 3.2}  & \mbox{$q$ is prime, $q\equiv \pm 1\pmod{16}$} & \sy 4\\
\hline
\mbox{Example 3.3} & \mbox{any} & \mbox{maximal parabolic}\\
\hline
\mbox{Example 3.4} & \mbox{$q$ is prime, $q\equiv \pm 1\pmod{16}$} & \sy 4\\
\hline
\mbox{Example 4.1} & \mbox{$q$ is prime, $q\equiv \pm 1\pmod{5}$} & \dih 5\\ 
\hline
\mbox{Example 5.3} & q\equiv \pm 1\pmod{8} & \sy 3\\
\hline
\end{array}
$$
\caption{The Li--Seress Examples}
\label{listable}
\end{table}
\end{center}

\begin{proof}
Suppose that $\Gamma$ is one of the $(G,2)$-arc-transitive examples
in~\cite{MR2258005} for a quasi\-primitive group $G$ of O'Nan--Scott type
{\sc Pa}. In each of the cases, 
$\soc G=M=T^k$ where $T\cong\psl 2q$ with some prime-power $q$. 
Let $\omega$ be a vertex of $\Gamma$.
The conditions for $q$ and for the projection 
$M_\omega\overline T/\overline T$ are 
in Table~\ref{listable}.
Suppose that $G\leq W$ is an inclusion as in~Hypothesis~\ref{hyp1}. 
By Corollary~\ref{pslcor}, $q\in\{5,7,9,11,19,29\}$. By Theorem~\ref{c2}, 
if $G\pi$ is intransitive or the type of the inclusion is $\cd{1S}$ or 
$\cd{2\sim}$, then $T=\alt 6\cong\psl 29$
and $M_\omega\overline T/\overline T\cong \dih 5$. However, this is 
not possible by Table~\ref{listable}. Inclusions of type $\cd 3$ 
are not possible by Theorem~\ref{c2}(3).
Thus the inclusion $G\leq W$ has type 
$\cd{2\not\sim}$. In this case, Lemma~\ref{notmax} implies  that $T=\psl 2q$
admits a factorization $T=AB$ with proper subgroups such that 
$M_\omega\overline T/\overline T\leq A\cap B$. Comparing the possibilities 
in Tables~\ref{excfact} and~\ref{listable} we find that this is impossible.
Hence no inclusion $G\leq W$ satisfying~Hypothesis~\ref{hyp1} is  possible.
\end{proof}

\newcommand{\etalchar}[1]{$^{#1}$}
\def\cprime{$'$}

\end{document}